\documentclass[12pt,leqno]{amsart}
\usepackage{amsthm, amsfonts, amssymb, amsmath,graphicx}

\def\cocoa
{\mbox{\rm C\kern-.13em o\kern-.07
em C\kern-.13em o\kern-.15em A}}

\def\cocoal
{\mbox{\rm C\kern-.13em o\kern-.07
em C\kern-.13em o\kern-.15em A\kern-.1em L}}

\def\opn#1#2{\def#1{\mathop{\kern0pt\fam0#2}\nolimits}} % to make operators

\newtheorem{theorem}{Theorem}[section]
\newtheorem{lemma}[theorem]{Lemma}

\newtheorem{proposition}[theorem]{Proposition}

\theoremstyle{definition}
\newtheorem{remark}[theorem]{Remark}
\newtheorem{definition}[theorem]{Definition}
\newtheorem{example}[theorem]{Example}
\newtheorem{remark/example}[theorem]{Remark/Example}

\let\oldlabel=\label
\def\prellabel{\marginparsep=1em\marginparwidth=44pt
 \def\label##1{\oldlabel{##1}\ifmmode\else\ifinner\else
 \marginpar{{\footnotesize\ \\ \tt
 ##1}}\fi\fi}}

\numberwithin{equation}{section}

\begin{document}
\title{ The standard graded property for vertex cover algebras of Quasi-Trees }
\author{Alexandru Constantinescu \and Le Dinh Nam}
\date{}
\maketitle
\begin{abstract}
In \cite{HHTZ} the authors characterize the  vertex cover algebras which are standard graded.  In this paper we give a simple combinatorial criterion for the standard graded property of vertex cover algebras in the case of quasi-trees. We also give an example of how this criterion works and compute the maximal degree of a minimal generator in that case.  
\end{abstract}
\section*{introduction}
Recently, the theory of vertex cover algebras has caught the attention of  researchers such as J. Herzog, T. Hibi and N.V. Trung. It has applications in graph and hypergraph theory. In some papers, as for instance in [HHT1],[HHT2],[HHTZ],[HHO], the authors study  the relation of this theory to unmixed bipartite graphs, perfect graphs and unimodular hypergraphs. The characterization of the standard graded property of vertex cover algebras is given in [HHTZ]. It asserts that the vertex cover algebra $A(\Delta)$ is standard graded if and only if $\Delta^{*}$ is a Mengerian simplicial complex. But it is not easy to check if the vertex cover algebra of a simplicial complex is standard graded or not. In this paper we add a new class of hypergraphs -- quasi-trees -- for which a simple combinatorial criterion for the standard graded property of the vertex cover algebra exists.

A quasi-tree is a connected simplicial complex whose facets can be ordered $F_1,\dots,F_m$ such that for all $i$, $F_i$ is a leaf of the simplicial complex with the facets $F_1,\dots ,F_i$. In Theorem 4.2 [HHTZ] the authors have  results about the standard graded property for a special quasi-tree class. The result states: a quasi-tree in codimension 1, such that each face of codimension 1 belongs to at most two facets, is a forest if and only if $A(\Delta)$ is standard graded. In this paper we  generalize  this result (Theorem \ref{thm1}). 
%We define a relation tree of $\Delta$ as a graph $T(\Delta)$ where the vertices of $T(\Delta)$ are the facets of $\Delta$ and the edges are obtained recursively as follows: choose a leaf $F$ of $\Delta$ and a branch $G$ of $F$ (the facet has property $ (H\cap F) \subset (G \cap F)$ for all $H \in \Im (\Delta)$), we denote this branch by $G=br_{T(\Delta)}(F)$. Then say $\{F,G\}$ is an edge of $T(\Delta)$. Remove $F$ from $\Delta$ and proceed with the remaining quasi-tree as before to determine the other edges of $T(\Delta)$. *\\
Our main theorem characterizes  the quasi-trees $\Delta$ for which $A(\Delta)$ standard graded. It asserts that $A(\Delta)$ is standard graded if and only if $\Delta$ has no special odd cycle. To illustrate our results, we give an example of a quasi-tree $\Delta$ for which $A(\Delta)$ is not standard graded and compute the maximal degree of a minimal generator in this case.\\

The authors wish to thank  J. Herzog and V. Welker for suggesting the problem and for many useful lessons, discussions and suggestions. Many thanks also to the organizers of Pragmatic 2008. 
\section{Quasi-Trees}
\begin{definition}
Let $ V = \{v_1,\dots,v_n\}$ be a finite set. A \textit{simplicial complex} $\Delta$ on $V$ is a collection of subsets of $V$ such that $F \in \Delta$ whenever $F \subset G$ for some $G \in \Delta$, and such that $\{v_i\} \in \Delta$ for $i=1,\dots,n$.\\
The elements of $\Delta$ are called \textit{faces}. The \textit{dimension of a face $F$}, dim$F$, is the number $\arrowvert F\arrowvert -1$. The \textit{dimension of $\Delta$} is \[\text{dim}\Delta=\text{max}\{\text{dim}F : F\in \Delta \}.\] The maximal faces under inclusion are called the \textit{facets} of the simplicial complex. The set of facets is denoted by $\Im (\Delta)$. A subcomplex $\Gamma$ of $\Delta$ is called a \textit{subcomplex of maximal dimension} (\emph{SMD} for short) if $\Im (\Gamma)\subset \Im (\Delta)$.
\end{definition}
\begin{definition}
Let $\Delta$ be a simplicial complex. A face $F \in \Im (\Delta)$ is called a \textit{leaf of $\Delta$} iff there exists $G \in \Im (\Delta)$ such that $ (H\cap F) \subset (G \cap F)$ for all $H \in \Im (\Delta)$. A face $G$ with the above property is called a \textit{ branch of $F$}.

Furthermore $\Delta$ is called a \textit{quasi-forest} if there exists a total order $\Im (\Delta)=\{F_1,\dots,F_m\}$ such that $F_i$ is a leaf of $\Gamma$, where  $\Gamma$ is the SMD  with $\Im (\Gamma)=\{F_1,\dots,F_i\}$ for all $i=1,\dots,m$. This order is called a \textit{leaf order} of the quasi-forest. A connected quasi-forest is called a \emph{quasi-tree}.\\
\end{definition}
\begin{definition}
Let $\Delta$ be a quasi-tree with leaf order $\{F_1,\dots,F_m\}$. We define \textit{a relation tree of $\Delta$}, denoted by $T(\Delta)$, in the following way:
\begin{itemize}
  \item  The vertices of $T(\Delta)$ are the facets of $\Delta$.
  \item  The edges are obtained recursively as follows: 
 \begin{itemize}
  \item[-] take the leaf $F_m$ of $\Delta$ and choose a branch $G$ of $F$;  denote  $G=br_{T(\Delta)}(F)$. 
  \item[-] Set $\{F_m,G\}$ to be an edge of $T(\Delta)$. 
  \item[-] Remove $F_m$ from $\Delta$ and proceed with the remaining quasi-tree as before to determine the other edges of $T(\Delta)$.
\end{itemize}
\end{itemize}   
\end{definition}

For the formulation of the following remark we say that a vertex $F$ of the graph $T(\Delta)$ is a  \textit{free vertex}  iff there exists only one vertex $G$ such that $\{F,G\}$ is an egde of $T(\Delta)$.
Note, that if the graph $T(\Delta)$ is considered as a 1-dimensional simplicial complex then a free vertex is a vertex in a leaf that is only contained in that leaf.
\begin{remark} (a) The graph $T(\Delta)$ depends on the leaf order and the choice of the branch for each leaf. However it is always a tree.\\
(b) Each free vertex of $T(\Delta)$ is a leaf of $\Delta$. 
\end{remark}
\begin{example}Let $\Delta$ be the  quasi-tree in Figure 1, which has  the facets $\{1,2,3\}$, $\{1,2,4\}$, $\{1,2,5\}$, $\{2,3,6\}$, $\{2,3,7\}$.\\

\begin{figure}[h]
\label{F1}
\setlength{\unitlength}{0.5cm}
\begin{picture}(12,8)
\put(1,3){\circle*{0.2}}
\put(2,7){\circle*{0.2}}
\put(4,1){\circle*{0.2}}
\put(6,6){\circle*{0.2}}
\put(8,1){\circle*{0.2}}
\put(10,7){\circle*{0.2}}
\put(11,3){\circle*{0.2}}

\put(1,3){\line(3,-2){3}}
\put(1,3){\line(5,3){5}}
\put(2,7){\line(1,-3){0.95}}
\put(2,7){\line(4,-1){4}}
\put(4,1){\line(2,5){2}}
\put(4,1){\line(1,0){4}}
\put(8,1){\line(-2,5){2}}
\put(8,1){\line(3,2){3}}
\put(10,7){\line(-4,-1){4}}
\put(10,7){\line(-1,-3){0.95}}
\put(11,3){\line(-5,3){5}}

\multiput(4,1)(-0.2,0.6){6}{\circle*{0.05}}
\multiput(8,1)(0.2,0.6){6}{\circle*{0.05}}

\put(2.1,2.7){$F_2$}
\put(3,5.5){$F_3$}
\put(5.6,2.6){$F_1$}
\put(8.2,5.5){$F_4$}
\put(9.3,2.7){$F_5$}

\put(3.7,0.2){1}
\put(5.8,6.3){2}
\put(7.9,0.2){3}
\put(0.3,2.8){4}
\put(1.6,7.3){5}
\put(10,7.3){6}
\put(11.3,2.8){7}

\end{picture}
\caption{$\Delta$}
\end{figure}

It is easy to check that both  graphs in Figure 2 are relation trees of $\Delta$, for the leaf order $\{F_1, F_2, F_3, F_4, F_5\}$.
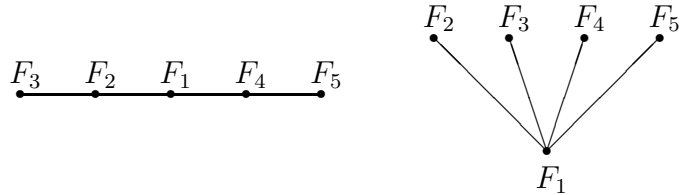
\begin{figure}[h]
\setlength{\unitlength}{0.5cm}
\begin{picture}(19,8)
\multiput(1,3.5)(2,0){5}{\circle*{0.2}}
\multiput(12,5)(2,0){4}{\circle*{0.2}}
\put(15,2){\circle*{0.2}}

\put(1,3.5){\line(1,0){8}}
\put(15,2){\line(-1,1){3}}
\put(15,2){\line(-1,3){1}}
\put(15,2){\line(1,3){1}}
\put(15,2){\line(1,1){3}}

\put(0.7,3.8){$F_3$}
\put(2.7,3.8){$F_2$}
\put(4.7,3.8){$F_1$}
\put(6.7,3.8){$F_4$}
\put(8.7,3.8){$F_5$}

\put(11.7,5.3){$F_2$}
\put(13.7,5.3){$F_3$}
\put(15.7,5.3){$F_4$}
\put(17.7,5.3){$F_5$}
\put(14.7,1){$F_1$}

\end{picture}
\caption{Two different relation trees of $\Delta$}
\end{figure}

\end{example}

Let $\Delta$ be a quasi-tree. Given $\{F_1,\dots,F_m\}$ a leaf order of $\Delta$ and a relation tree $T(\Delta)$, we define  a partial order on $\Im (\Delta)$ in the following way: \[G\leq_{T(\Delta)}F\] iff there exists a chain $G=H_0,H_1,\dots,H_t=F$ such that $ H_i=br_{T(\Delta)}(H_{i+1})$ for all $i=0,\dots,t-1$. Because $\Delta$ is a quasi-tree, there always exists smallest  facet $F_1$ in $\Im (\Delta)$ with respect to this order. For this facet define: $br_{T(\Delta)}(F_1):=F_1$. We will omit the index $T(\Delta)$ in the notation of the partial order wherever $T(\Delta)$ is clear from the context.

It is natural   for a simplicial complex to call an alternating sequence of distinct vertices and facets $v_1,F_1,v_2,F_2,\dots, v_s,F_s,v_{s+1}=v_1,$ with $s\geq2$, a \textit{cycle} if $v_i,v_{i+1}\in F_i$ for all $i=1,\dots,s$. Such a cycle is called \textit{special} if no facet contains more than two vertices of the cycle. We have the following result:
\begin{lemma}
\label{lm2}
Let $\Delta$ be a quasi-tree, $ v_1,F_1,v_2,F_2,\dots,v_s,F_s,v_{s+1}=v_1$ be a cycle of $\Delta$ and $T(\Delta)$  be a relation tree of $\Delta$. Denote by $T_0$ the minimal subtree of $T(\Delta)$ with respect to inclusion containing $\{F_1,\dots,F_s\}$. Then
$ \mid G\cap \{v_1,\dots,v_s\} \mid \geq 2$ for all  $ G \in V(T_0)$.
\end{lemma}
\begin{proof}
Let $\Gamma$ be the quasi-tree with $\Im (\Gamma)=V(T_0)$. Consider on $\Gamma$ the leaf order induced by the leaf order on $\Delta$: $\Im (\Gamma)=\{G_1,\dots,G_r\}$. We have the partial order $\leq$ on $\Im(\Delta)$ induced by the leaf order and by the relation tree $T(\Delta)$ as above. Because $T_0$ is the minimal subtree of $T(\Delta)$ containing $\{F_1,\dots,F_s\}$ we have that if $H\in V(T_0)$, with $H\leq F_i$ for all $i=1,\dots,s,$ then $H=G_1$.\\
Let $G\in V(T_0)$.\\

 If $G\not= G_1$ then there exists a set $A=\{i_1,\dots,i_p\} \subset \{1,\dots,s\}$ such that $G\leq F_i$ for all $i\in A$ and $G\not\leq F_j$ for all $j\in B$ with $B=\{1,\dots,s\}\setminus A$. Assume $i_1<i_2< \dots< i_p$.
\begin{itemize}
  \item[-] If there  exists $q\in \{1,\dots,p-1\}$ such that $i_{q+1}-i_q\geq 2$ then \[\{v_{i_q+1},v_{i_{q+1}}\}\subset \big(\bigcup_{i\in A}F_i\big)\bigcap\big(\bigcup_{j\in B}F_j\big).\]
  \item[-] If $i_1,\dots,i_p$ are consecutive then \[\{v_{i_1},v_{i_p+1}\}\subset \big(\bigcup_{i\in A}F_i\big)\bigcap\big(\bigcup_{j\in B}F_j\big).\]
\end{itemize}
So we have \[ \mid \big(\bigcup_{i\in A}F_i\big)\bigcap\big(\bigcup_{j\in B}F_j\big) \cap \{v_1,\dots,v_s\} \mid \geq 2.\] We will show that \[\Big( \big(\bigcup_{i\in A}F_i\big)\bigcap\big(\bigcup_{j\in B}F_j\big) \cap \{v_1,\dots,v_s\}\Big)\subset G.\]\\
Let \[v_i\in \big(\bigcup_{i\in A}F_i\big)\bigcap\big(\bigcup_{j\in B}F_j\big).\] So $v_i\in F_{i-1}$ and $v_i\in F_i$, with $i-1\in A, i\in B$ or $i\in A, i-1\in B$. We can assume without loss of generality that $i\in A, i-1\in B$.\\ 
Because $G\leq F_i$ there exists a chain $G = H_1, H_2, \dots, H_t = F_i$ such that  $H_i=br_{T(\Delta)}(H_{i+1})$. As we also have $G\not\leq F_{i-1}$ there exists another chain $K_1,\dots,K_l=F_{i-1}$ such that $K_i=br_{T(\Delta)}(K_{i+1})$ and $K_1$ appears before G in the leaf order $\{F_1,\dots,F_m\}$. Choose this last chain such that $K_1$ is maximal with this property (with respect to the leaf order). By the choice of branches  we have $v_i\in H_j$ for all $j=1,\dots,t$ and $v_i\in K_j$ for all $j=2,\dots,l$. So $v_i\in G$.\\

 If $G=G_1$ we have two cases:

\textit{Case 1:} There exists only one $H\in V(T_0)$ such that $G_1=br_{T(\Delta)}(H)$ (in this case $H=G_2$). So $G_2\leq G_t$ for all $t=3,\dots,r$. By the minimality of $T_0$ we  have $G_1\in \{F_1,\dots,F_s\}$. So  $\mid G\cap \{v_1,\dots,v_s\} \mid \geq 2$.

\textit{Case 2:} There exists a set of indices $J=\{j_1,j_2,\dots,j_k\}\subset \{3,\dots,r\}$ such that:
\begin{itemize}
\item[-] $G_1=br_{T(\Delta)}(G_{j})$ for all $j\in J$, 
\item[-] $G_1\not\in \{F_1,\dots,F_s\}$ and  
\item[-] $G_2\not\leq G_j$ for all $j\in J$. 
\end{itemize}
There exists a set $A=\{i_1,\dots,i_p\} \subset \{1,\dots,s\}$ such that:
\begin{itemize}
\item[-] $G_2\leq F_i$ for all $i\in A$ and 
\item[-]$G_2\not\leq F_l$ for all $l\in B$, with $B=\{1,\dots,s\}\setminus A$. 
\end{itemize}
For all $l\in B$ there exists $j\in J$ such that $G_j\leq F_l$. As above, there exist two vertices $\{v_{\alpha},v_{\beta}\}\subset (G_2\cap(\bigcup_{j\in J} G_j))$. Because $G_1=br_{T(\Delta)}(G_j)$ for all $j\in J$ we also have $\{v_{\alpha},v_{\beta}\}\subset G_1$.   
\end{proof}
\section{The vertex cover algebra of a quasi-tree}
\begin{definition}
Let $\Delta$ be a simplicial complex on the vertex set $[n]$. An integer vector $a=(a_1,\dots,a_n)\in \mathbb{N}^n$ is called a \textit{cover of order $k$ (or k-cover) of $\Delta$} if $\sum_{i\in F}a_i\geq k$ for all $F \in \Im (\Delta)$.
\end{definition}
Let $K$ be a field and $S=K[x_1,\dots,x_n]$ be the polynomial ring in $n$ indeterminates over $K$. Let $S[t]$ be a polynomial ring over $S$ in the indeterminate $t$ and consider the $K$-vector space $A(\Delta)\subset S[t]$  generated by all monomials $x_1^{a_1}x_2^{a_2}\dots x_n^{a_n}t^k$ such that $a=(a_1,\dots,a_n)\in \mathbb{N}^n$ is a k-cover of $\Delta$.

We have $A(\Delta)=\bigoplus_{k\geq 0}A_k(\Delta)$ with $A_0(\Delta)=S$ and $A_k(\Delta)$ is spanned by the monomials $x_1^{a_1}x_2^{a_2}\dots x_n^{a_n}t^k$ such that $a=(a_1,\dots,a_n)\in \mathbb{N}^n$ is a k-cover of $\Delta$. If $a$ is a $k$-cover and $b$ is a $l$-cover, then $a + b$ is a $(k + l)$-cover. This implies that
$$A_k(\Delta)A_l(\Delta)\subset A_{k+l}(\Delta).$$
Therefore, $A(\Delta)$ is a graded $S$-algebra. We call it the \textit{vertex cover algebra} of the simplicial complex $\Delta$.

A k-cover $a=(a_1,\dots,a_n)\in \mathbb{N}^n$ of a simplicial complex is \textit{decomposable}
if there exists an $i$-cover $b$ and a $j$-cover $c$ \big($b,c\not=(0,\dots,0)$\big) such that $a=b+c$ and $k=i+j$. If $a$ is not decomposable, we call it \textit{indecomposable}.\\

For $F\subset [n]$ we denote by $P_F$ the prime ideal generated by the variables $x_i$ with $i\in F$. \\
Let $\Delta$ be a simplicial complex, we set \[I=\bigcap_{F\in \Im (\Delta)}P_F.\] The symbolic Rees algebra of $I$ is: \[R^s(I)=\bigoplus_{k\geq0}I^{(k)}t^k.\] In \cite[Lemma 4.1]{HHT1}, the authors prove that $A(\Delta)=R^s(I)$  and in Theorem 4.2 they describe the structure of a vertex cover algebra. Here is the precise result:
\begin{theorem}[Theorem 4.2 \cite{HHT1}]
The vertex cover algebra $A(\Delta)$ is a finitely generated, graded and normal Cohen-Macaulay $S$-algebra.
\end{theorem}
\begin{definition}
Let A be a finitely generated $\mathbb{N}$-graded algebra. Denote by $d(A)$ the maximal degree of generator of A in a minimal generating set.  The algebra A is called \textit{standard graded} iff $d(A)=1$.
\end{definition}
\begin{remark}
If $\gamma$ is an indecomposable k-cover of $A(\Delta)$ then $\gamma$ belongs to the set of generators of $A(\Delta)$. In particular, $d(A(\Delta))\geq k$.
\end{remark}
\begin{lemma}
\label{lm1}
Let $\Delta$ be a simplicial complex and let $F$ be a leaf of $\Delta$. Let $\Gamma$ be the SMD obtained from $\Delta$ by removing $F$. We have that $d(A(\Delta)) \geq d(A(\Gamma))$.
\end{lemma}
\begin{proof}
Since $F$ is a leaf of $\Delta$ , there must exist  a vertex contained in  $F$  which does not belong to any other facet of $\Gamma$. Assume $V(\Gamma)=\{1,\dots,m\}$ and $m+1 \in F\setminus V(\Gamma)$.\\
For all $k>0$, if $c=(c_1,\dots,c_m,0,\dots,0)$ is an indecomposable k-cover of $\Gamma$, then $c'=(c_1,\dots,c_m,k,0,\dots,0)$ is an indecomposable k-cover of $\Delta$. So we have $d(A(\Delta)) \geq d(A(\Gamma))$.
\end{proof}
In \cite{HHTZ}, the authors characterized the simplicial complexes $\Delta$ such that $A(\Gamma)$ is standard graded for all subcomplexes of maximal dimension $\Gamma$ of $\Delta$. The precise result is:
\begin{theorem}[ Theorem 2.2 \cite{HHTZ}]
\label{bs}
Let $\Delta$ be a simplicial complex. The following conditions are equivalent:
\begin{itemize}
\item[(i)] The vertex cover algebra $A(\Gamma)$ is standard graded for all $\Gamma \subseteq \Delta$.
\item[(ii)] The vertex cover algebra $A(\Gamma)$ has no generator of degree 2 for all $\Gamma \subseteq \Delta$.
\item[(iiii)] $\Delta$ has no special odd cycle.
\end{itemize}
\end{theorem}
For a quasi-tree, we would like to have a more precise characterization. We will prove the following:
\begin{theorem}
\label{thm1}
Let $\Delta$ be a quasi-tree. The following conditions are equivalent:
\begin{itemize}
\item[(i)] The vertex cover algebra $A(\Gamma)$ is standard graded for all $\Gamma \subseteq \Delta$.
\item[(ii)] The vertex cover algebra $A(\Delta)$ is standard graded.
\item[(iii)] $\Delta$ has no special odd cycle.
\end{itemize}
\end{theorem}
\begin{proof}
From Theorem \ref{bs} we know:
\begin{itemize}
\item[-] (i) if and only if (iii).
\item[-] (iii) implies (ii).
\end{itemize}   
So, to complete the proof,  we only need to show that (ii) implies (iii).\\
Assume $\Delta$ contains a special odd cycle $ v_1,F_1,v_2,F_2,\dots,v_{2k+1},F_{2k+1},v_1$. We need to prove that $d(A(\Delta))>1$. As above, let $T(\Delta)$ be a relation tree of $\Delta$ and let $T_0$ be the minimal subtree of $T(\Delta)$,  with respect to inclusion, containing $\{F_1,\dots,F_{2k+1}\}$. Let $\Gamma$ be the SMD corresponding to $T_0$. Set $ r=|\Delta|-|\Gamma|$. We  use induction on $r$.\\

 If $r=0$ then $\Delta=\Gamma$.\\
Put $a_i = \left\{ \begin{array}{ll}
1 & \textrm{if $i=v_j$, for some $j=1,\dots,2k+1$}\\
0 & \textrm{otherwise}
\end{array} \right.$ \\
and $\delta=(a_1,\dots,a_n)$. By Lemma \ref{lm2} we have that $\delta$ is a 2-cover of $\Delta$. As $ v_1,F_1,v_2,F_2,\dots,v_{2k+1},$ $F_{2k+1},v_1$ a special odd cycle, $\delta$ is an indecomposable 2-cover of $\Delta$. So $\Delta$ is not 
standard graded, a contradiction.\\

 If $r>0$ we have two cases.

\textit{Case 1:} There exists a leaf $F$ of $\Delta$ not belonging to $\Gamma$. Let $\Sigma$ be the SMD obtained from $\Delta$ by removing $F$, so $\Sigma$ contains $\Gamma$. By Lemma \ref{lm1} we have $d(A(\Delta)) \geq d(A(\Sigma))$ and by induction on $r$ we have $d(A(\Sigma))>1$, so $d(A(\Delta))>1$.

\textit{Case 2:} All leaves of $\Delta$ belong to $\Gamma$. We  show that $\Delta=\Gamma$. As all leaves of $\Delta$ belong to $\Gamma$, every free vertex of $T(\Delta)$ belongs to $T_0$. Because $T(\Delta)$ is a tree, it does not contain any cycles. If we extend $T_0$ then we always obtain a tree that has at least one free vertex that does not belong to $V(T_0)$. So we cannot extend $T_0$ to obtain $T(\Delta)$. This means that $T(\Delta)=T_0$. So we have that $\Delta=\Gamma$. As $r>0$ this is  a contradiction.
\end{proof}
\begin{example}
\label{vd}
Let $n\geq3$ and $V=\{1,\dots,2n\}$. Let $\Delta_n$ be the quasi-tree  on $V$ with facets $\{F,F_1,\dots,F_n\}$ where $F_i=\{1,\dots,\hat{i},\dots,n,n+i\}$ for all $i=1,\dots,n$ and $F=\{1,\dots,n\}$. The vertex cover algebra of this quasi-tree is not standard graded because it contains the special odd cycle $2,F_1,3,F_2,1,F_3,2$.
\end{example}
We would like to have a more precise result for this example. We have the following:
\begin{proposition} Let $\Delta_n$ be the simplicial complex defined above. Then   $d(A(\Delta_n))=n-1$.
\end{proposition}
\begin{proof}
First, we show there exists an indecomposable $(n-1)$-cover. Set $a=(\underbrace{1,1,\dots,1}_{n},\underbrace{0,0,\dots,0}_{n})$. Obviously $a$ is $(n-1)$-cover. Assume $a$ is decomposable, that is $a=b+c$ with $b=(b_1,\dots,b_{2n})$ a $r$-cover and $c=(c_1,\dots,c_{2n})$ a $s$-cover. If the number of $b_i\not= 0$ is $k$ then the number of $c_j\not= 0$ is $n-k$. We have $r\leq k-1$ and $s\leq (n-k)-1$ so $r+s<n-1$, a contradiction. So $a$ is indecomposable.\\

Second, we show that any $k$-cover $a=(a_1,\dots,a_{2n})$ with $k\geq n$ is decomposable.

If there exists $j\in [n]$ such that $a_j\not=0$ and $a_{n+j}\not=0$, we have 
\begin{eqnarray*}
(a_1,\dots,a_{2n})&=&\phantom{+}(a_1,\dots,a_j-1,\dots,a_{n+j}-1,\dots,a_{2n})+ \\
  &&+  (0,\dots,0,\underbrace{1}_{j},0,\dots,0,\underbrace{1}_{n+j},0,\dots,0)
\end{eqnarray*}
where  the first summand   is a $(k-1)$-cover and
the second one  is  a 1-cover. So $a$ is decomposable.
%where  $(a_1,\dots,a_i-1,\dots,a_{n+i}-1,\dots,a_{2n})$   is a $(k-1)$-cover and\\
%$(0,\dots$, $0$ , $\underbrace{1}_{i},$$0,$$\dots,0,\underbrace{1}_{n+i},0,\dots,0)$ is  a 1-cover. So $a$ is decomposable.\\

 If there exists $i\in [n]$ such that $a_i=0$ and $a_{n+i}\not=0$, we have 
$$(a_1,\dots,a_{2n})=(a_1,\dots,a_{n+i}-1,\dots,a_{2n})+(0,\dots,0,\underbrace{1}_{n+i},0,\dots,0).$$
As  $a_1+\dots+a_{i-1}+a_{i+1}+\dots+a_n\geq k$,   the first summand   is a $k$-cover.
The second one  is  a 0-cover. So $a$ is decomposable.\\
%with $(0,\dots,0,\underbrace{1}_{n+i},0,\dots,0)$ is a 0-cover and $(a_1,\dots,a_{n+i}-1,\dots,a_{2n})$ is a $k$-cover (because $a_1+\dots+a_{i-1}+a_{i+1}+\dots+a_n\geq k$). So $a$ is decomposable.

So we can suppose $a_{n+i}=0$ for all $i=1,\dots,n$. We can assume that $a_1\leq a_2\leq\dots\leq a_n$. So $a$ is a $k$-cover iff $a_1+a_2\dots+a_{n-1}\geq k$. We have two cases:

\textit{Case 1:} $a_1=a_2=\dots=a_n=p$ with $p\geq 2$. Then 
$$a=(\underbrace{1,\dots,1}_{n},\underbrace{0,\dots,0}_{n})+(\underbrace{p-1,\dots,p-1}_{n},\underbrace{0,\dots,0}_{n})$$
with the first summand   a $(n-1)$-cover and the second summand  a $k-(n-1)$-cover. So $a$ is decomposable.
%with $(\underbrace{p-1,\dots,p-1}_{n},\underbrace{0,\dots,0}_{n})$  a $k-(n-1)$-cover and 
%$(\underbrace{1,\dots,1}_{n},\underbrace{0,\dots,0}_{n})$  a $(n-1)$-cover. So $a$ is decomposable.

\textit{Case 2:} There exist $r\in[n]$ such that \[a_1\leq a_2\leq \dots\leq a_r < a_{r+1}=a_{r+2}=\dots=a_n.\] 
We have
\begin{eqnarray*}
a&=&\phantom{+}(0,\dots,0,\underbrace{1,\dots,1}_{n-r},\underbrace{0,\dots,0}_{n})+\\
&&+(a_1,\dots,a_r,a_{r+1}-1,\dots,a_n-1,0,\dots,0).
\end{eqnarray*}
Obviously $(0,\dots,0,\underbrace{1,\dots,1}_{n-r},\underbrace{0,\dots,0}_{n})$ is a $(n-r-1)$-cover. Because 
\[\begin{array}{l}
a_1\leq\dots\leq a_r\leq a_{r+1}-1\leq\dots\leq a_n-1\textrm{~ and} \\
\rule{0pt}{3ex}a_1+\dots+a_r+a_{r+1}-1+\dots+a_{n-1}-1\geq k-(n-r-1)
\end{array}\]
 we have that $(a_1,\dots,a_r,a_{r+1}-1,\dots,a_n-1,0,\dots,0)$ is a $(k-n+r+1)$-cover. So $a$ is decomposable.
\end{proof}
The following figure shows $\Delta_n$ for $n=3$:
\begin{figure}[h]
\setlength{\unitlength}{1cm}
\begin{picture}(6,6)
\put(1,1){\circle*{0.1}}
\put(3,1){\circle*{0.1}}
\put(5,1){\circle*{0.1}}
\put(2,3){\circle*{0.1}}
\put(4,3){\circle*{0.1}}
\put(3,5){\circle*{0.1}}

\put(1,1){\line(1,2){2}}
\put(3,5){\line(1,-2){2}}
\put(1,1){\line(1,0){4}}
\put(2,3){\line(1,-2){1}}
\put(2,3){\line(1,0){2}}
\put(3,1){\line(1,2){1}}
\put(0.7,1){$5$}
\put(2.9,0.55){$3$}
\put(5.1,1){$4$}
\put(1.7,3){$1$}
\put(4.1,3){$2$}
\put(2.9,5.2){$6$}
\put(2.8,3.5){$F_3$}
\put(2.85,2.1){$F$}
\put(1.8,1.6){$F_1$}
\put(3.8,1.6){$F_2$}

%\put(2,-1){\textsc{Figure} 3. $\Delta_3$} 
\end{picture}
\caption{$\Delta_3$}
\end{figure}

\vspace{1cm}
Address of the authors:\\
\textsc{Dipartimento di Matematica,
Universit\`a di Genova\\
Via Dodecaneso 35, 16146 Genova, Italia}\\ 
\textit{E-mail:} \texttt{constant@dima.unige.it,     ledinh@dima.unige.it}

\end{document}